\documentclass[11pt]{amsart}
\usepackage{fullpage} 
\usepackage{amsmath}
\usepackage{amssymb}
\usepackage{stmaryrd}
\usepackage{setspace}
\usepackage{nicefrac}
\usepackage{xypic}
\usepackage{multicol}
\usepackage{amsthm}
  \usepackage[pdftex,
  bookmarks = false,
  pdfstartview = FitBH,
  linktocpage = true,
  pagebackref = true,
  pdfhighlight = /O,
  colorlinks=true,
  linkcolor=blue,
  citecolor=blue,
  filecolor = blue,
  urlcolor = blue,
  menucolor = blue,
]{hyperref}
\usepackage{empheq}

\input xy 
\xyoption{all} 

\newtheorem{thm}{Theorem}

\theoremstyle{definition}
\newtheorem{defn}[thm]{Definition}

\newtheorem{rmk}[thm]{Remark}

 \newcommand{\ql}{\ulcorner}
 \newcommand{\qr}{\urcorner}

 \newcommand{\SM}[1]{\underline{#1}}

\DeclareMathOperator{\Con}{Con}

\title{Yablo's Paradox And Arithmetical Incompleteness}
\author{Graham Leach-Krouse}

\begin{document}
\maketitle
\begin{abstract}
In this short paper, I present a few theorems on sentences of arithmetic which are related to Yablo's Paradox\footnotemark~as G\"odel's first undecidable sentence was related to the Liar paradox.\footnotemark~In particular, I consider two different arithemetizations of Yablo's sentences: one resembling G\"odel's arithmetization of the Liar, with the negation outside of the provability predicate, one resembling Jeroslow's undecidable sentence,\footnotemark~with negation inside. Both kinds of arithmetized Yablo sentence are undecidable, and connected to the consistency sentence for the ambient formal system in roughly the same manner as G\"odel and Jeroslow's sentences.

Finally, I consider a sentence which is related to the Henkin sentence ``I am provable'' in the same way that first two arithmetizations are related to G\"odel and Jeroslow's sentences. I show that this sentence is provable, using L\"ob's theorem, as in the standard proof of the Henkin sentence.
\end{abstract}
 
 \footnotetext[1]{G\"odel makes this connection in his published presentation of his result---\cite[p149]{Godel:2011vh}.}
 \footnotetext[2]{See \cite{Yablo:1993vh} for a canonical presentation} 
 \footnotetext{presented in \cite{Jeroslow:1973tz}}
 
Below, I present a few theorems on sentences of arithmetic which are related to Yablo's Paradox approximately as G\"odel's first undecidable sentence was related to the Liar paradox. Each one involves taking a sentence---G\"odel's, Jeroslow's, or Henkin's, as the superscripts suggest---which is normally thought of as self-referential, and ``unfurling'' it into a sequence of sentences which would naturally be thought of as constituting a non-well-founded chain of reference.
 
These results are of interest, first, because they tend to support G\"odel's assertion that any ``epistemological antinomy'' can be used as a guide in finding some form of mathematical incompleteness,\footnote{\cite[n14]{Godel:2011vh}} and second, because they present a striking link between the behavior of the arithmetical sentences mentioned above (G\"odel's, Henkin's, and Jeroslow's) and the behavior of their respective unfurled counterparts considered below. In the absence of a counterexample, the correspondence suggests that the pattern may extend to other cases, or may even be itself a candidate for investigation, should a reasonable formalization of the notion of unfurling become available.

\section{preliminaries}
The technology employed by the proofs below is almost entirely standard. We do require free-variable versions of the diagonal lemma, and of Hilbert-Bernays derivability conditions. The generalized derivability conditions (GD1),(GD2), (GD3) may be somewhat unfamiliar. The last two can be found in \cite{Feferman:1960tk}, and the first is easy to prove in any reasonable system. We let $\mathsf{M}$ denote any r.e. system of reasonable power (if you like, $\mathsf{PA}$).
\begin{defn} We let $\Box$ abbreviate a predicate satisfying
\begin{enumerate}
\item[(GD1)] $\mathsf{M}\vdash \phi(x) \Rightarrow \mathsf{M}\vdash \Box(\ql\phi(\dot{x})\qr)$
\item[(GD2)] $\mathsf{M}\vdash \Box(x\SM{\rightarrow} y)\rightarrow( \Box(x)\rightarrow \Box(y))$
\item[(GD3)] $\psi(y_1,\ldots,y_k)\in \Sigma_1\Rightarrow \mathsf{M}\vdash \psi(y_1,\ldots,y_k)\rightarrow\Box(\ql \psi (\dot{y}_1,\ldots,\dot{y}_k)\qr)$
\end{enumerate}
\end{defn}

\begin{rmk}[Generalized L\"ob's Theorem]\label{GLT} It follows from (GD1)-(GD3) that $M\vdash \Box(\ql\phi(\dot{x})\qr)\rightarrow\phi(x)$ implies $\mathsf{M}\vdash \phi(x)$. 
\end{rmk}
\begin{defn}
With $k$ a free variable, let
\begin{align*}
\mathsf{M}\vdash Y^J(k) &\leftrightarrow (\forall x > k)[\Box(\ql \neg Y(\dot{x})\qr)]\\
\mathsf{M}\vdash Y^G(k) &\leftrightarrow(\forall x > k)[\neg\Box(\ql Y(\dot{x})\qr)]\\
\mathsf{M}\vdash Y^H(k) &\leftrightarrow (\forall x > k)[\Box(\ql Y(\dot{x})\qr)]
\end{align*}
\end{defn}
\begin{rmk}\label{GTL} Inspecting the definitions, we see that if $x> y$, then $\mathsf{M}\vdash Y^{J/G/H}(\bar{y})\rightarrow Y^{J/G/H}(\bar{x})$.
\end{rmk}
\section{Theorems}
\begin{thm} ~
\begin{enumerate}
\item For any $k$ 
\begin{enumerate}
\item If $1$-$\Con(\mathsf{M})$, then $\mathsf{M}\nvdash  Y^J(\bar{k})$.
\item If $\Con(\mathsf{M})$, then $\mathsf{M}\nvdash \neg Y^J(\bar{k})$
\end{enumerate}
\item For any $k$
\begin{enumerate}
\item If $\Con(\mathsf{M})$, then $\mathsf{M}\nvdash Y^G(\bar{k})$
\item If $1$-$\Con(\mathsf{M})$, then $\mathsf{M}\nvdash \neg Y^G(\bar{k})$
\end{enumerate}
\item $\mathsf{M}\vdash Y^H(k)$
\end{enumerate}
\end{thm}
\begin{proof} For (1), first suppose that $\mathsf{M}\vdash Y^J(\bar{k})$. Then, evidently,  $\mathsf{M}\vdash \Box(\ql \neg Y^J(\dot{\bar{k}}+1)\qr)$, so by $1$-$\Con(\mathsf{M})$, $\mathsf{M}\vdash \neg Y^J(\bar{k}+1)$, which contradicts $Y^J(\bar{k})$, by Remark \ref{GTL}. Now, suppose that $\mathsf{M}\vdash \neg Y^J(\bar{k})$. This implies $\mathsf{M}\vdash (\exists x > \bar{k})[\neg\Box(\ql \neg Y(\dot{x})\qr)]$, which violates the second incompleteness theorem.

For (2), first suppose that $\mathsf{M}\vdash Y^G(\bar{k})$. Then, evidently $\mathsf{M}\vdash \neg\Box(\ql Y(\dot{\bar{k}}+1)\qr)$, But, also $\mathsf{M}\vdash Y^G(\bar{k}+1)$,  by Remark \ref{GTL}, and so $\mathsf{M}\vdash \Box(\ql Y^G(\dot{\bar{k}}+1)\qr)$, which violates the consistency of $\mathsf{M}$. On the other hand, suppose $\mathsf{M}\vdash \neg Y^G(\bar{k})$. Then $\mathsf{M}\vdash (\exists x > k)[\Box(\ql Y^G(\dot{x})\qr)]$. By $1$-consistency, and $\Sigma_1$-completeness we get that, for some $x$, $\mathsf{M}\vdash\Box(\ql Y^G(\dot{\bar{x}})\qr)$, and by a second application of $1$-consistency, we get that $\mathsf{M}\vdash Y^G(\bar{x})$, which is impossible, by the first part of the argument.

For (3), we aim to show that $\mathsf{M}\vdash \Box (\ql Y^H(\dot{k})\qr)\rightarrow Y^H(k)$, and appeal to the generalized L\"ob's theorem of Remark \ref{GLT}. So, assume, in $\mathsf{M}$, that $\Box(\ql Y^H(\dot{k})\qr)$, so that by (GD1),(GD2),
\begin{equation}
\Box(\ql (\forall x > \dot{k})[\Box(\ql Y^H(\dot{x})\qr)]\qr).\tag{$\gamma$}
\end{equation} In $\mathsf{M}$, fix an arbitrary $x>k$. By (GD3), and since ``$x>k$'' is a $\Sigma_1$ formula, we have $ \Box(\ql \dot{x}>\dot{k}\qr)$. Thus, by ($\gamma$) and (GD1),(GD2), we have $\Box(\ql (\forall z > \dot{x})[\Box(\ql Y^H(\dot{z})\qr)]\qr)$. So equivalently, by (GD1), (GD2), $\Box(\ql Y^H(\dot{x})\qr))$. Since $x>k$ was arbitrary, we have $(\forall x>k)[\Box(\ql Y^H(\dot{x})\qr)]$, which implies $Y^H(k)$, Discharging our assumption, we have $\mathsf{M}\vdash \Box(\ql Y^H(\dot{k})\qr)\rightarrow Y^H(k)$, whence, by the generalized L\"ob Theorem \ref{GLT}, $\mathsf{M}\vdash Y^H(k)$.
\end{proof}

\begin{thm} Let $k$ be a free variable. Then
$$\mathsf{M}\vdash \Con(\mathsf{M})\leftrightarrow Y^G(k)$$
\end{thm} 
\begin{proof}

The right-to-left implication is clear. For left-to-right, formalize the argument of (2) above:

That is, in $\mathsf{M}$, let $x>k$ be arbitrary. Then, aiming at a refutation of $\Con(\mathsf{M})$, assume in $\mathsf{M}$ that $\Box(\ql Y^G(\dot{x})\qr)$. We then have, by (GD1), (GD2), that $\Box(\ql\neg\Box(\ql Y^G(\dot{\dot{x}}+1)\qr)\qr)$, but also $\Box(\ql Y^G(\dot{x}+1)\qr)$, from which, by (GD3), we have $\Box(\ql\Box(\ql Y^G(\dot{x}+1)\qr)\qr)$. These two together imply $\neg \Con(\mathsf{M})$. So, discharging our assumption and contraposing, $\Con(\mathsf{M})\rightarrow \neg\Box(\ql Y^G(\dot{x})\qr)$. As $x>k$ was arbitrary, we have $(\forall x>k)[\Con(\mathsf{M})\rightarrow  \neg\Box(\ql Y^G(\dot{x})\qr)]$, so by standard manipulation of quantifiers, $\Con(\mathsf{M})\rightarrow (\forall x>k)[\neg\Box(\ql Y^G(\dot{x})\qr)]$. So, $\Con(\mathsf{M})\rightarrow Y^G(k)$.
\end{proof}

\begin{thm} Let $k$ be a free variable. Then,
$$\mathsf{M}\vdash \Con(\mathsf{M})\leftrightarrow \neg Y^J(k)$$
\end{thm}
\begin{proof} The right-to-left implication is clear. For left-to-right,  let $x>k$ be arbitrary in $\mathsf{M}$, and assume in $\mathsf{M}$ that $\Box(\ql \neg Y^J(\dot{x})\qr)$. Then, by (GD1), (GD2), $\Box(\ql (\exists y > \dot{x})[\neg\Box(\ql \neg Y(\dot{y})\qr)]\qr)$, whence $\Box\Con(\mathsf{M})$. By a formalized version of G\"odel's second theorem, $\Box(\ql\Con(\mathsf{M})\qr)\rightarrow\neg\Con(\mathsf{M})$, so $\neg \Con(\mathsf{M})$ follows. Thus, discharging our assumption and contraposing, $\Con(\mathsf{M})\rightarrow \neg\Box(\ql Y^J(\dot{x})\qr)$. Since $x>k$ was arbitrary, evidently, $\Con(\mathsf{M})\rightarrow (\forall x>k)[\neg\Box(\ql \neg Y^J(\dot{x})\qr)]$, so very directly, $\Con(\mathsf{M})\rightarrow (\exists x>k)[\neg\Box(\ql \neg Y^J(\dot{x})\qr)]$, whence $\Con(\mathsf{M})\rightarrow \neg Y^J(k)$.
\end{proof}
\section{acknowledgements}
Thanks to Rafal Urbaniak and Cezary Cieslinski for some suggested simplifications, particularly for the observation that  a standard diagonal lemma, and 1-consistency were sufficient for these results. Thanks to Chris Porter for invaluable editorial advice. 
\bibliographystyle{alpha}
\bibliography{/applications/TeX/bibliographies/papers}

\end{document}